\documentclass[12pt]{amsart}
\newtheorem{theorem}{Theorem}[section]

\newtheorem{remark}[theorem]{Remark}
\newtheorem{corollary}[theorem]{Corollary}

\newtheorem{proposition}[theorem]{Proposition}

\newtheorem{definition}[theorem]{Definition}

\usepackage{enumerate}
\usepackage{amsfonts}
\usepackage{amssymb}
\usepackage{amsthm}
\usepackage{amsmath}

\usepackage[all]{xy}
\usepackage[dvips,final]{graphics}
\numberwithin{equation}{section}

\begin{document}

\title{Characterizing slices for proper actions of locally compact groups  }

\author{Sergey A. Antonyan}
\address{Departamento de  Matem\'aticas,
Facultad de Ciencias, Universidad Nacional Aut\'onoma  de M\'exico,
 04510 Ciudad de M\'exico, M\'exico.}
\email{antonyan@unam.mx}

\begin{abstract} In his seminal work \cite{pal:61}, R. Palais extended  a
substantial part of the theory of compact transformation
groups to the case of proper actions of locally compact groups.  Here we extend  to proper actions some other important results well known for compact group actions. In particular, we prove that if $H$ is a compact subgroup of a locally compact group $G$ and $S$ is   a small (in the sense of Palais) $H$-slice in a proper $G$-space, then the action map $G\times S\to G(S)$ is open.  This is applied to prove that the slicing map $f_S:G(S)\to G/H$ is continuos and open, which provides an external characterization of a slice. Also an equivariant extension theorem is proved for proper actions. As an application, we give  a short proof of the compactness of the Banach-Mazur compacta.

\end{abstract}

\thanks{The author was supported in part  by grant IN-115717 from PAPIIT (UNAM)}

\maketitle
\markboth{SERGEY A. ANTONYAN}{Characterizing slices for proper actions}

\section{Introduction}
The letter  $G$ will denote a Hausdorff topological group with  unit element $e	\in G$.
All spaces are assumed to be completely regular and  Hausdorff.

By an action of  $G$ on a   space $X$ we mean a
continuous map $(g, x)\mapsto gx$ of the product $G\times X$ into
$X$ such that $ex=x$ and $(gh)x=g(hx)$, whenever $x\in X$, $g,
h\in G$ and $e$ is the unity of $G$.   A space $X$ together with a
fixed action of the group $G$ is called a $G$-space.

If $X$ and $Y$ are  $G$-spaces, then a continuous map $f:X\to Y$
is called a $G$-map or an equivariant map, if $f(gx)=gf(x)$ for
every $x\in X$ and $g\in G$. For a point  $x\in X$, the subgroup
$G_x=\{g\in G\mid gx=x\}$ of $G$ is called the {\em stabilizer} or
{\em isotropy subgroup} at $x$. Clearly, $G_x\subset G_{f(x)}$
whenever $f$ is a $G$-map and $x\in X$.

If $X$ is a $G$-space, then for a subset $S\subset X$ and a subgroup $H\subset G$, the $H$-hull (or $H$-saturation) of $S$ is defined as follows:  $H(S)$= $\{hs \ |\ h\in H,\ s\in S\}$. If $S$ is the  one point set
$\{x\}$, then the $H$-hull $H(\{x\})$ usually is denoted by $H(x)$  and called  the $H$-orbit of $x$. The orbit
space $X/H$ is always considered in its quotient topology.
A subset $S\subset X$ is called $H$-invariant or,  if it coincides with
its $H$-hull, i.e., $S=H(S)$.
 A $G$-invariant set is also
called, simply, invariant.

For a closed subgroup $H \subset G$, by $G/H$ we will denote the $G$-space
of cosets $\{xH ~| \ x\in G\}$ under the action induced by left translations, i.e., $g(xH)=(gx)H$ whenever $g\in G$ and $xH\in G/H$.

\medskip

One of the most fundamental facts in the theory of $G$-spaces when $G$ is
a compact Lie group is, the so called, Slice Theorem. In its full generality it 
was proved by G.Mostow \cite{mostow} (see also \cite[Corollary 1.7.19]{pal:60}).

\begin{theorem}[Slice Theorem]\label{T:slice}
 Let $G$ be a compact Lie group and $X$ a  $G$-space. Then for every point $x\in X$, there exists a $G_x$-slice $S\subset X$ such that $x\in S$.
\end{theorem}

We recall  the  definition of an $H$-slice (cf. \cite[\S 1.7]{pal:60}).

\begin{definition}\label{D:slice} Let $X$ be a $G$-space and  $H$ a closed subgroup of $G$.  A subset $S\subset  X$ is called an $H$-slice in $X$, if:
 \begin{enumerate}
 \item $ S$ is H-invariant, i.e., $H(S) = S$,
 \item  $S$ is closed in $G(S)$,
\item  if $G \setminus  H$, then $gS\cap S = \emptyset$,
 \item  the saturation $G(S)$ is open in $X$.
  
\noindent 
If in addition $G(S) = X$, then we say that $S$ is a global $H$-slice of $X$.
 \end{enumerate}
\end{definition}

To each $H$-slice $S\subset X$, a $G$-map $f_S:G(S)\to G/H$, called the {\it slicing map}, is associated according to the following rule: 
$$f_S(gs)=gH\quad \text{for every} \quad  g\in G, \ s\in S.$$
Lets check that that $f_S$ is well defined.
Indeed,   if $gs=g's'$ for some $g, g'\in G$ and $s,s' \in S$ then 
   $s =g^{-1}g's' \in S \cap g^{-1}g'S$.  
  Then item (3) of Definition \ref{D:slice} yields that   $g^{-1}g' \in H$ which is equivalent to 
  $ gH = g'H$, as required. 
  Thus,  the slicing map $f_S$ is well defined. 
  
  It is immediate that the slicing map is equivariant, i.e., $f_S(gx) = gf_S(x)$  for all $ x \in  G(S)$  and $ g \in  G$.
  It is also clear that $S=f^{-1}(eH)$, where $eH\in G/H$ stands for the coset of the unit element $e\in G$.  
       
It is a well known fact that the slicing map is continuous whenever the acting group $G$ is compact (see e.g., \cite[Theorem 1.7.7]{pal:60}); this  gives  the following  important external characterization of an $H$-slice. 

\begin{theorem}\label{charatercompact} Let $G$ be a compact group, $H$ a closed subgroup of $G$, and $X$ a $G$-space. Then there exists a   one-to-one correspondence between all equivariant maps $f:X\to G/H$ and global $H$-slices $S$ in $X$ given by $f\mapsto S_f:=f^{-1}(eH)$. The inverse correspondence is given by $S\mapsto f_S$, above defined.
\end{theorem}

Below, in Theorem \ref{T:slicing}
 we generalize Theorem \ref{charatercompact} to the case of proper actions of arbitrary locally compact groups, which are  an important generalization of actions of compact groups.
This result is proceeded by Theorem \ref{P:open} which establishes some important properties of small global slices. Then these results are applied in Section \ref{S:orbit} to  orbit spaces    of proper $G$-spaces. In Section \ref{S:extension} we prove an equivariant extension theorem which is further applied to get an equivariant extension of a continuous map defined on a small cross section. In the final Section \ref{S:BM}  the results of Sections \ref{S:slice} and \ref{S:orbit} are applied to give a short proof of the compactness of the Banach-Mazur compacta $BM(n)$, $n\ge 1$.

Recall that the concept of a {\it proper action} of a  locally compact  group was introduced in  1961 in the seminal work of R. Palais~\cite{pal:61}.
 This notion allowed R. Palais to extend  a
substantial part of the theory of compact Lie transformation
groups to non-compact  ones.  
Perhaps, some detailed definitions are in order here.

 Let $G$ be a locally compact group and  $X$  a $G$-space. Two subsets  $U$ and $V$ in  $X$  are called  thin relative to each other
  \cite[Definition 1.1.1]{pal:61},  if the set
  $$\langle U,V\rangle=\{g\in G \mid  gU\cap V\ne \emptyset\}$$ called {\it the transporter} from $U$ to $V$,
has  compact closure in $G$.
   A subset $U$ of a $G$-space $X$ is called {\it $G$-small}, or just {\it small}, \ if  every point in $X$ has a neighborhood thin relative to $U$. 
      
   \begin{definition}[\cite{pal:61}] Let $G$ be a locally compact group. 
   A $G$-space $X$
is called  {\it  proper},  if   every point in  $X$ has a small neighborhood.
\end{definition}

Each orbit in a proper $G$-space is closed, and each stabilizer is compact \cite[Proposition
1.1.4]{pal:61}. It is easy to check that: (1) the product of two $G$-spaces
is proper whenever one of them is so; (2) the inverse image of a proper $G$-space under a $G$-map is
again a proper $G$-space.

Important examples of  proper $G$-spaces are the coset spaces
$G/H$ with $H$ a compact subgroup of a locally compact group $G$.
Other interesting examples the reader can find in 
\cite{ab:78}, \cite{ant:05},  \cite{kos:65} and \cite{pal:61}.

\

 Among other important results,   Palais  proved in \cite{pal:61}
  the following generalization of the Slice Theorem.
 
 \begin{theorem}[Slice Theorem for proper actions]\label{PalaisSlice} 
 Let $G$ be a  Lie group (not necessarily compact) and $X$ a  proper $G$-space. Then for every point $x\in X$, there exist an invariant neighborhood $U$ of $x$ and an equivariant map $f:U\to G/G_x$ such that $x\in f^{-1}(eG_x)$. 
\end{theorem}   

 Since the distinguished point $eG_x$ is a global $G_x$-slice for the proper $G$-space $G/G_x$, and since  the inverse image of a slice is again a slice (see \cite[Corollary 1.7.8]{pal:60}), we infer that 
  $S=f^{-1}(eG_x)$ is a $G_x$-slice (in the sense of Definition \ref{D:slice}). Besides, $S$ is a small subset of its saturation $G(S)=U$ since $eG_x$ is a small subset of $G/G_x$.
 Therefore, at the first glance, it may seem that Theorem \ref{PalaisSlice}   establishes something more than the existence of a slice.
  However, this is not the case as it follows from our Theorem \ref{T:slicing} below.
 
 \
 
 \section{Some important properties of small slices}\label{S:slice}
 
 In this section we will prove  the following two main results.
      
\begin{theorem}\label{P:open} Let $G$ be a locally compact group,  $X$  a proper $G$-space, $H$ a compact subgroup of $G$, and $S$ a global $H$-slice of
$X$ which is a small subset. Then 
\begin{enumerate}
\item the restriction $f:G\times S\to X$  of the action is an open map.
\item the restriction $p:S\to X/G$ of the orbit map $X\to X/G$ is an open  map.\end{enumerate}
\end{theorem}

\begin{proof} (1) Let $O$ be an open subset of $G$ and $U$ be an open subset of $S$.  It suffices to show that the set $OU=\{gu~|~g\in O, \ u\in U\}$ is open in $X$.

Define $W=\bigcup\limits_{h\in H}(Oh^{-1})\times (hU)$.
Observe that
$$X\setminus OU=f\bigl((G\times S)\setminus W\bigr).$$

Indeed, since $OU=f(W)$ \ and \ $X=f(G\times S)$, the inclusion  $X\setminus OU\subset f\bigl((G\times S)\setminus W\bigr)$ follows.

Let us establish  the converse inclusion
$f\bigl((G\times S)\setminus W\bigr)\subset X\setminus OU$.

Assume the contrary, that there exists a point  $gs\in f\bigr(G\times S)\setminus W\bigr)$ with  $(g, s)\in (G\times S)\setminus W$ such that  $gs\in OU$. Then $gs=tu$ \ for some \ $(t, u)\in O\times U$. Denote $h=g^{-1}t$. One has  
$$s=g^{-1}tu=hu$$
and
$$(g, \ s)=(tt^{-1}g, \ g^{-1}tu)=(th^{-1}, \ hu)\in (O h^{-1})\times (hU).$$
  Since both $s$ and $u$ belong to $S$,  and  $s=hu$, by item (3) of Definition \ref{D:slice}, we conclude that  $h\in H$. Consequently, $(O h^{-1})\times (hU)\subset W$, yielding that    $(g, s)\in W$, a contradiction. Thus, the equality $X\setminus OU=f\bigl((G\times S)\setminus W\bigr)$ is proved.

Now we observe that  $(G\times S)\setminus W$  is  closed  in  $G\times S$, and hence, in  $G\times X$.
  Since $S$ is a closed small subset of $X$,  by \cite[Proposition 1.4(c)]{ab:78},  $f$ is a closed map. This yields that   the set
$f\bigl((G\times S)\setminus W\bigr)$  is closed,  and hence, $OU$ is open in $X$, as required.
\smallskip

(2) Indeed, let $U$ be an open subset of $S$. By item (1) of this theorem, the saturation $G(U)$ is open in $X$  yielding  that the intersection $G(U)\cap S$ is open in $S$. Since 
  $p^{-1}\big(p(U)\big)=G(U)\cap S$ we conclude that $p(U)$ is open in $X/G$, as required.
\end{proof}
 
  
 \medskip

This therem is now applied to give an external characterization of a slice in a proper $G$-space.

 \begin{theorem}\label{T:slicing}
  Let $G$ be a locally compact group,  $H$ a compact subgroup of $G$.
 Let  $X$ be a proper  $G$-space
  and  $S$ a global $H$-slice of  $X$ which is a small subset in $X$. Then the slicing map  $f_S:X \to G/H$ is continuous  and open.  If, in addition, $S$ is compact then $f_S$  is also closed.
   
    Conversly, if one has an equivariant map $f:X \to G/H$, then the inverse image $S=f^{-1}(eH)$ is a  global $H$-slice which is a small subset of   $X$, and $f_S=f$.    
      \end{theorem}

\begin{proof} 
  Let  $\alpha:G \times S \to X$ be the restriction of the action   $G\times X\to  X$ and let  
  $\pi:{G \times S} \to G $ denote the  projection.
  
  The quotient map  $p:{G } \to G/H, \ p(g)= gH$, is  open  (and closed since $H$ is compact) and it makes the following diagram commutative:
  \[ \xymatrix{G \times S \ar[r]^-\pi \ar[d]^\alpha & G  \ar[d]^p
                                                       \\
               X \ar[r]^f & G/H.
               }   \]
  Since $S$ is a closed small subset of $X$,  Theorem  \ref{P:open}(1) yields that  $\alpha$ is an open map. Since  $\pi$ and $p$ are continuous, the  equality  $f \alpha = p \pi$ implies that $f$ is continuous. 

Since the maps $\pi$ and $p$ are open and $\alpha$ is continuous, we infer that $f$ is also open.
 If, in addition, $S$ is compact then the map $\pi$ is also closed, which yields that $f$ is closed.

The converse assertion is immediate since the point $eH\in G/H$ is a small global $H$-slice for $G/H$ and an inverse image of a small global $H$-slice is so (see \cite[p.\,10]{pal:61}). The equality   $f_S=f$ is 
a simple verification.
\end{proof}

\medskip

Since each compact subset of a proper $G$-space is a small subset \cite[p.\,300]{pal:61}, Theorem \ref{T:slicing} has the following immediate corollary.

\begin{corollary}\label{C:slice}  Let $G$ be a locally compact group,  $H$ a compact subgroup of $G$.
 Let  $X$ be a proper  $G$-space
  and  $S$ a compact global $H$-slice of  $X$. Then the slicing map  $f_S:X \to G/H$
 is continuous,  open and closed.
      \end{corollary}

      \begin{remark}\label{R:small}
Theorem \ref{T:slicing} is not valid if the $H$-slice $S\subset X$ is not a small subset. Here is a simple counterexample. 
Let $G=\mathbb R_{+}$, the multiplicative group of the positive reals and  $X=\mathbb R^2\setminus\{0\}$, the Euclidean plane without the origin.  Consider   the action $G\times X\to  X$  defined by means of the ordinary scalar multiplication, i.e., if $\lambda \in G$ and $A=(x, y)\in X$, then  $\lambda *A :=\lambda A=(\lambda x, \lambda y)$. It is easy to see that this  action is proper. Further, let 
$$S:=\big\{(x, \pm \frac{1}{x}) \mid x\in \mathbb R\setminus \{0\}\big\} \cup \big\{(0, \pm 1), (\pm 1, 0)\big\}.$$
 Then, clearly,  $S$ is a global $H$-slice of $X$ with $H=\{1\}$, the trivial subgroup of $G$ while  $S$ is not a small subset of $X$. Also it is easy to see that the corresponding slicing map $f_S:X\to G$ is not continuous.

It is interesting to notice that the unit circle $S=\{A\in X\mid \Vert A\Vert =1\}$ is also a global  $H$-slice for $X$. However, in this case the slicing map $f_S:X\to G$  is continuous because $S$, being compact, is a small subset and then Corollary \ref{C:slice} applies. Moreover, in this case  $f_S(A)=\Vert A\Vert $, which clearly,    is a continuous $G$-map.
\end{remark}

           \

          \section{Orbit spaces}\label{S:orbit}
          
          Existence of slices facilitates the study of
transformation groups since, for example, it enables the reduction of global
questions about transformation groups to local ones. On the other hand, existence of global slices in proper $G$-spaces enables the reduction of studying  the orbit space of a non-compact group action  to that of a compact subgroup.
          
                    The following  result for $G$ a compact group can be found in  \cite[Proposition 1.7.6]{pal:60}.
          
   \begin{theorem}\label{Orbitspace}
  Let $G$ be a locally compact group,  $H$ a compact subgroup of $G$.
 Let  $X$ be a proper  $G$-space
  and  $S$ a small global $H$-slice of  $X$. 
  Then the inclusion $S\hookrightarrow X$ induces a homeomorphism of the orbit spaces 
   $S/H$ and $X/G$.
    \end{theorem} 
       \begin{proof} 
    Let $p:S\to X/G$ be the restriction of the orbit projection $X\to X/G$.
    Then,  according to Theorem \ref{P:open},  $p$ is continuous and  open. Since $p$ is 
 constant on the $H$-orbits of $S$, it induces a continuous open map $p':S/H\to X/G$.
 
 It remains to show that  $p'$ is a bijection. Indeed, if $G(x)\in X/G$ is any point then $x=gs$ for some $g\in G$ and $s\in S$  since $S$ is a global $H$-slice. Then, clearly, $p(s)=G(s)=G(x)$ showing that $p'\big(H(s)\big)=G(x)$.  Thus, $p'$ is surjective.

To see that $p'$ is injective, assume that $p'\big(H(s)\big)=p'\big(H(s_1)\big)$ for some $s, s_1\in S$. Then $G(s)=G(s_1)$, yielding that $s=gs_1$ for some $g\in G$. But then, by item (3) of Definition \ref{D:slice},  we get that $g\in H$ which implies that $H(s)=H(gs_1)=H(s_1)$, as desired. This completes the proof that $p':S/H\to X/G$ is a bijection, and hence, a  homeomorphism.
             \end{proof}

     Since every compact subset of a proper $G$-space is small \cite[p.\,300]{pal:61}, Theorem \ref{Orbitspace} 
      has the following immediate corollary.
          
       \begin{corollary}\label{C:Orbitspace} Let $G$ be a locally compact group,  $H$ a compact subgroup of $G$.  Let  $X$ be a proper  $G$-space   and  $S$ a compact global $H$-slice of  $X$. 
  Then the inclusion $S\hookrightarrow X$ induces a homeomorphism of the orbit spaces 
   $S/H$ and $X/G$.
    \end{corollary}

    It turns out that  in the presence of compactness of the global $H$-slice $S$,  the other assumptions in Theorem \ref{Orbitspace} may essentially be weakened. Namely, the following version of Theorem \ref{Orbitspace} holds true.

   \begin{proposition}\label{P:Orbitspace}
  Let $G$ be any topological group and  $H$ a closed subgroup of $G$.
 Let  $X$ be a  $G$-space
  and  $S$ a compact global $H$-slice of  $X$. 
  Then the inclusion $S\hookrightarrow X$ induces a homeomorphism of the orbit spaces 
   $S/H$ and $X/G$ provided that $X/G$ is  Hausdorff.
    \end{proposition} 
       \begin{proof} 
    Let $p:S\to X/G$ and $p':S/H\to X/G$ be as in the proof of Theorem \ref{Orbitspace}.
    Since $p'$ is a continuous bijection, it remains to show that it is a closed map. But this 
    is due to  the hypotheses since  $S/H$ is compact and $X/G$ is Hausdorff. 
             \end{proof}

           \medskip

    \section{Extension to an equivariant map}\label{S:extension}
                
             Let $G$ be a locally compact group and  $H\subset G$ a compact subgroup.   
                If $X$ is a proper $G$-space and $S$  a global $H$-slice of $X$, then it is well known (cf. \cite[Proposition 2.1.3]{pal:61})  that any $H$-equivariant map $f:S\to Y$ uniquely extends to a $G$-equivariant map $F:X\to Y$. This result can be generalized in the following manner (for compact group actions it was proved in \cite[Ch.\,I,  Theorem 3.3]{bre:72}).
                
 \begin{theorem}\label{T:section}
  Let $G$ be a locally compact group acting properly on the space $X$, and let $Y$ be any $G$-space. 
Let $S$  be any closed small subset of $X$, and let $f: S\to  Y$ be a continuous map such that whenever
$s$ and $gs$  are both in $S$ (for some $g\in  G$), then $f(gs) = gf(s)$. Then
$f$ can be extended uniquely to an equivariant map $F: G(S)\to  Y$.
    \end{theorem}       
   
   \begin{proof}     For any $g\in G$ and $s\in S$,  put $F(gs) = gf(s)$.  To see that $F$  is well
defined let $gs = g's'$. Then $s = (g^{-1}g')s' $ so that $f(s) = f((g^{-1}g')s')
= g^{-1}g'f(s')$, by assumption. 
   Thus, $gf(s) = g'f(s')$, as desired.
   
    To see that
$F$  is continuous, let $(x_i)$ be a net in $G(S)$  converging to $ x\in G(C)$. Then
$x_i = g_is_i$ and $x=gs$ for some $s, s_i\in S$ and $g, g_i\in G$.  Thus,  
    \begin{equation}\label{1}
    g_is_i \rightsquigarrow gs.
    \end{equation}
    
     Since $S$ is a small set, we can choose a neighborhood $U$ of $s$ such that the transporter $\langle S, U\rangle$ has compact closure in $G$. Then, by convergence, there exists an index $i_0$ such that 
    $(g^{-1}g_i)s_i\in U$ whenever $i\ge i_0$. Hence $g^{-1}g_i\in \langle S, U\rangle$ for $i\ge i_0$. Since  $\langle S, U\rangle$ has compact closure in $G$, by passing to a subnet we may assume that      
$
    g^{-1}g_i \rightsquigarrow h
$
   for some $h\in G$. 
    Then $g_i \rightsquigarrow gh$ and 
        \begin{equation}\label{3}
    g_i^{-1} \rightsquigarrow h^{-1}g^{-1}.
        \end{equation}    

Now,  (\ref{1}) and (\ref{3}) imply that $s_i \rightsquigarrow h^{-1}s$. Since $s_i\in S$, by closedeness of $S$ we conclude that $ h^{-1}s\in S$. Then by the hypothesis   we have  $f(h^{-1}s)=h^{-1}f(s)$ and by continuity of $f$  we get that 
$f(s_i)\rightsquigarrow f(h^{-1}s)=h^{-1}f(s).$ Consequently, 
$$F(x_i)=F(g_is_i)=g_if(s_i)\rightsquigarrow gh(h^{-1}f(s))=gf(s)=F(gs)=F(x).$$
This proves the continuity of $F$, as desired.
           \end{proof}

          \medskip
          
          Recall that a continuous map $s:X/G\to X$ is called a cross section for the orbit map $\pi: X\to X/G$ if the composition $\pi s$ is the identity map of $X/G$. It is easy to see that the image $C:=s(X/G)$ is closed in $X$ (since $X$ is Hausdorff). It turns out that if, in addition, $C$ is any small subset of $X$, then it uniquely determines the cross section. Because of this fact, we shall  use the term \lq\lq cross section\rq\rq \ for the closed  image of a cross section.

         More precisely,  we have the following result. 
          
 \begin{proposition}\label{P:small}
  Let $X$ be a proper $G$-space with $G$ a locally compact group and let $\pi :X\to X/G$ be the orbit map. Let $C$ be a small closed
subset of $X$ touching each orbit in exactly one point. Then the map $s: X/G\to X$ defined by $s(\pi(x)) = G(x) \cap  C$ is a cross section. Conversely, the image
of any cross section is closed in $ X.$
 \end{proposition}
\begin{proof} We need to show that $s$  is continuous. For this let $A\subset C$ be
closed. Since $C$ is a small closed subset, one can apply \cite[Proposition 1.4(c)]{ab:78} according to which  the set   $s^{-1}(A) = G(A)$ is closed in $X$, as desired. 
 
 For the converse, let  
$C = s(X/G)$ and let $(x_i)$ be a net in $C$ converging to $x\in  X$. 
  We have $\lim p(x_i)= p(x)$ and  $\lim s\big(p(x_i)\big)= s\big(p(x)\big)$. Therefore, $x=\lim x_i =\lim s\big(p(x_i)\big)=s\big(p(x)\big)\in C$ showing that $x\in C$. Thus,  $C$  is closed.

 \end{proof}

          Theorem \ref{T:section} has the following interesting corollary.
           
 \begin{corollary} Let $G$ be a locally compact group, $X$ a proper $G$-space and  $Y$  any $G$-space. Assume that  $C\subset X$ is a closed small cross section of the orbit map $p:X\to X/G$. Then each continuous map  $f: C\to  Y$  such that $G_c\subset G_{f(c)}$ for all $c\in C$, has a  unique extension  to an equivariant map $F: X\to Y$. 
  \end{corollary}
  
  \begin{proof}  
 If $c$ and $gc$ belong to $C$ for some $g\in G$, then $gc=c$ since $C$ touchs the orbit $G(c)$ in exactly one point. Thus, $g\in G_c$. Since $G_c\subset G_{f(c)}$, we get that  $gf(c)=f(c)=f(gc)$. Thus, the hypotheses of Theorem \ref{T:section} are  fulfilled, and hence, its application completes the proof.
 
      \end{proof}
  
 \begin{remark} The example of the proper $G$-space $X$ in Remark \ref{R:small} shows that in Proposition \ref{P:small} one cannot omit the smallness condition on the set $C$.
 \end{remark}

 

 \section{An Application}\label{S:BM}
 In this section we apply Corollary \ref{C:Orbitspace}  to give a short proof of the compactness of the Banach-Mazur compacta $BM(n)$, $n\ge 1$.

  As usual, for an integer $n\ge 1$, 
  $\mathbb R^n$ denotes the $n$-dimensional Euclidean space with the standard norm, and 
   $GL(n)$ denotes the real full linear group.
We denote by $\mathcal B(n)$ the hyperspace of all compact convex
bodies  of $\mathbb R^n$  with odd symmetry about the origin, equipped with the Hausdorff metric 
$$d_H(A,B)=\max\left\{ \sup\limits_{b\in B}d(b,A), ~ \sup\limits_{a\in A} d(a, B)\right\},$$
where $d$ is the standard Euclidean metric on $\Bbb R^n$.

We consider  the natural  action of $GL(n)$ on  $\mathcal B(n)$  defined as follows:
$
(g, A)\longmapsto gA; \quad gA=\{ga \mid  \ a\in A\}, \ \ \text{for all}
 \  \ g\in GL(n), \ \ A\in \mathcal B(n).
$

In \cite{ant:Bull} it was proved that the $GL(n)$-space  $\mathcal N(n)$  consisting of all norms $\varphi: \mathbb R^n\to \mathbb R$, endowed with the compact-open topology and the natural action of $GL(n)$, is a proper  $GL(n)$-space.  
Using the fact that $\mathcal B(n)$ is $GL(n)$-equivariantly homeomorphic  to  $\mathcal N(n)$ (see \cite[p.\,210]{ant:00}),  we infer that  $\mathcal B(n)$ is a proper  $GL(n)$-space \cite{ant:00}.  A direct  way of proving the properness of the $GL(n)$-space $\mathcal B(n)$ one can find in  \cite[Theorem 3.3]{antjo:13}.

  According to a  theorem of F.~John \cite{john},  for any $A\in \mathcal B(n)$, there is unique maximal volume   ellipsoid $j(A)$ contained in $A$ (respectively, minimal volume ellipsoid $l(A)$ containig $A$). Usually, $j(A)$ is called  the John ellipsoid of $A$,  and $l(A)$ is called the  L\" owner ellipsoid of $A$.

This fact allows to define  two  {\it special} \ global $O(n)$-slices in $\mathcal B(n)$, where $O(n)$     denotes the orthogonal subgroup of $GL(n)$.
 
Denote  by $J(n)$ the subset of $\mathcal B(n)$ consisting of all bodies $A\in \mathcal B(n)$ for which the ordinary Euclidean unit ball 
 $\mathbb B^n=\big\{(x_1,\dots, x_n)\in\mathbb R^{n} \  \big | \ \sum_{i=1}^nx_i^2\leq 1\big\} 
$
 is the maximal volume  ellipsoid {\it contained} in $A$. In \cite[Theorem 4]{ant:00} it was proved that $J(n)$ 
is {\it a global} \ $O(n)$-slice for $\mathcal B(n)$.

Analogously, the subspace $L(n)$ of $\mathcal B(n)$ consisting of all bodies $A\in \mathcal B(n)$ for which  $\mathbb B^n$ is the minimal volume  ellipsoid {\it containing}  $A$, is {\it a global} \ $O(n)$-slice for $\mathcal B(n)$. 

Lets give a direct proof  that $J(n)$ and $L(n)$ are compact.

\begin{proposition}\label{P:J(n)compact}
 $J(n)$ is compact.

\end{proposition}
\begin{proof} It is known  that there is a closed ball $D\subset \mathbb R^n$ centered at the origin such that $A\subset D$ for every $A\in J(n)$. Moreover, it was proved by F. John \cite{john} that the radio of $D$ may be taken even $\sqrt n$.

 Thus $J(n)$ is a subset of the hyperspace $cc(D)$ of all non-empty compact convex  subsets of $D$ endowed with the Hausdorff metric topology.
Since  $cc(D)$ is compact (in fact,  it is  homeomorphic to the Hilbert cube \cite[Theorem 2.2]{Nadler}), it suffices to show that $J(n)$ is  closed in  $cc(D)$. But this is evident since, if 
$(A_k)_{k\in\mathbb N}\subset J(n)$ is a sequence converging to  $A\in cc(D)$, then  $A$ should contain the unit ball $\mathbb B^n$  since every $A_k$ does. Hence, $A$ has non-empty interior, i.e., it  is a convex body, and then,  $A\in \mathcal B(n)$. Now we apply \cite[Theorem 4(4)]{ant:00} according to which $J(n)$ is closed in  $\mathcal B(n)$. This yields that $A\in J(n)$, and hence, $J(n)$ is closed in $cc(D)$, as desired.

\end{proof}

In a similar way,  one can prove the compactness of the global $O(n)$-slice $L(n)$.  Here one should take into account that every $A\in L(n)$ contains the closed ball of radio $\frac{1}{\sqrt n}$ centered at the origin of $\mathbb R^n$ (see \cite[p.\,559]{john}).

\smallskip

Thus, $\mathcal B(n)$ is a proper  $GL(n)$-space and $J(n)$ is a global $O(n)$-slice of it (see  \cite[Theorem 4]{ant:00}). Since by Proposition \ref{P:J(n)compact}, $J(n)$ is compact,
 Corollary \ref{C:Orbitspace}  immediately yields the following corollary (cf. \cite[Corollary 1]{ant:00}). 

\begin{corollary}\label{C:1}  The orbit space $\mathcal B(n)/GL(n)$
 is homeomorphic to the $O(n)$-orbit space $J(n)/O(n)$.
\end{corollary}

Corollary \ref{C:slice} and Proposition \ref{P:J(n)compact} yield the following corollary.

  \begin{corollary} The slcing map $f_{J(n)}:\mathcal B(n)\to GL(n)/O(n)$ corresponding to the global $O(n)$-slice $J(n)$ is continuous, open and closed.  
\end{corollary}

The compactness of the orbit space $\mathcal B(n)/GL(n)$ originally was established in \cite{Macbeath}. Corollary \ref{C:1} and Proposition   \ref{P:J(n)compact} immediately yield an alternative and short proof of the compactness of  $\mathcal B(n)/GL(n)$, known as the Banach-Mazur compactum $BM(n)$ (see \cite{ant:00}).            
                          
  \begin{corollary} $\mathcal B(n)/GL(n)$ is a compact metrizable space. 
\end{corollary}                    

\

\bibliographystyle{amsplain}

\bibliography{triquot}

\end{document}